%% file: main.tex
\documentclass[10pt, a4paper,
oneside, headinclude,footinclude]{scrartcl}
\usepackage[utf8]{inputenc}
\pdfoutput=1 
\input{structure.tex} 

\title{\normalfont\spacedallcaps{Intrinsically Lipschitz functions with normal target in Carnot groups}} 
\author{\spacedlowsmallcaps{Gioacchino Antonelli\textsuperscript{*} and Andrea Merlo\textsuperscript{**}}}

\date{}

\begin{document}

\renewcommand{\sectionmark}[1]{\markright{\spacedlowsmallcaps{#1}}} 
\lehead{\mbox{\llap{\small\thepage\kern1em\color{halfgray} \vline}\color{halfgray}\hspace{0.5em}\rightmark\hfil}} 
\pagestyle{scrheadings}
\maketitle 
\setcounter{tocdepth}{2}
\paragraph*{Abstract}
We provide a Rademacher theorem for intrinsically Lipschitz functions $\phi:U\subseteq \mathbb W\to \mathbb L$, where $U$ is a Borel set, $\mathbb W$ and $\mathbb L$ are complementary subgroups of a Carnot group, where we require that
$\mathbb L$ is a normal subgroup. Our hypotheses are satisfied for example when $\mathbb W$ is a horizontal subgroup. Moreover, we provide an area formula for this class of intrinsically Lipschitz functions.  
\paragraph*{Keywords} Carnot groups, intrinsically Lipschitz functions, Rademacher theorem, area formula.
\paragraph*{MSC (2010)} 53C17, %   Sub-Riemannian geometry
	%53C60,   % Finsler spaces and generalizations
	% 53C30,  % Homogeneous manifolds
	22E25, % Nilpotent and solvable Lie groups
	28A75,  %  Length, area, volume, other geometric measure theory
	%49N60, % Regularity of solutions 
	49Q15, %  Geometric measure and integration theory, integral and normal currents
	%53C38% Calibrations and calibrated geometries
	%58C35 % Integration on manifolds; measures on manifolds
	26A16.  % Lipschitz (Hlder) classes
	%26B20 Integral formulas (Stokes, Gauss, Green, etc.)
	%54Exx, % Spaces with richer structures 
	%37L40 %Invariant measures
	%58D05, %Groups of diffeomorphisms and homeomorphisms as manifolds
	%22F50, %Groups as automorphisms of other structures
	% 22DXX % Locally compact groups and their algebras
	% 22F30. % Homogeneous spaces
	%14M17. %Homogeneous spaces and generalizations (within Algebraic geometry)
	% 53C30 % Homogeneous manifolds
	% 58D19 % Group actions and symmetry properties
	% 58C25 % Differentiable maps

%\tableofcontents

{\let\thefootnote\relax\footnotetext{* \textit{Scuola Normale Superiore, Piazza dei Cavalieri, 7, 56126 Pisa, Italy,}}}
{\let\thefootnote\relax\footnotetext{** \textit{Universit\'a di Pisa, Largo Bruno Pontecorvo, 5, 56126 Pisa, Italy.}}}

\section{Introduction}
In the last years there has been an increasing interest towards Geometric Measure Theory in the non-smooth setting, and in particular in the setting of sub-Riemannian Carnot groups for what concers the notion of rectifiability. This line of research was initiated by the result in \cite{AK00} in which the authors proved that the first Heisenberg group $\mathbb H^1$ is not $k$-rectifiable, according to Federer's definition (see \cite{Federer1996GeometricTheory}), for $k\geq 2$.

In this note we focus on the notions of intrinsically Lipschitz function and intrinsically differentiable function in the setting of Carnot groups. We refer the reader to
\cite{FranchiSerapioni16} for a wide study of the notion of intrinsically Lipschitz function and to \cite{ChousionisFasslerOrponen19} for some recent developments.
The notion of intrinsically differentiable function has been introduced in \cite{FSSC06} and then widely studied in the last years: we refer the reader to \cite{ASCV06, ArenaSerapioni, FSSC11, FMS14, ADDDLD20} for some developments. 

One of the main open questions in this area of research is whether a Rademacher type theorem holds. Namely, is it true that every intrinsically Lipschitz function between complementary subgroups of a Carnot group is intrinsically differentiable almost everywhere? Some first answers have been given in \cite{FMS14, FSSC11} in the setting of Carnot groups $\mathbb G$ of type $\star$, i.e., a class strictly larger than Carnot groups of step 2, and for maps $\phi:U\subseteq\mathbb W\to \mathbb L$, where $U$ is \textbf{open} and $\mathbb W$ and $\mathbb L$ are complementary subgroups of $\mathbb G$, with $\mathbb L$ \textbf{horizontal and one-dimensional}. 
Very recently, by making use of currents, the author of \cite{Vittone20} has proved the Rademacher theorem for intrinsically Lipschitz maps between complementary subgroups of any dimension in the Heisenberg groups $\mathbb H^n$.

In this note we prove that a Rademacher type theorem with \textbf{normal} target is true in arbitrary Carnot groups. 
We provide also an area formula for maps that satisfy our hypotheses. 
\begin{teorema}\label{thm:AbelianRademacherIntro}
Let $\mathbb W$ and $\mathbb L$ be two complementary subgroups of a Carnot group $\mathbb G$. 
Let us assume that $\mathbb L$ is \textbf{normal} and let $U\subseteq \mathbb W$ be a \textbf{Borel} set. Let $\phi: U\subseteq \mathbb W\to\mathbb L$ be an intrinsically Lipschitz function. Let $\Phi:U\to \mathbb G$ be the graph map of $\phi$, i.e., $\Phi(w):=w\cdot\phi(w)$ for every $w\in U$.

Then $\mathbb W$ is a \textbf{Carnot subgroup} of $\mathbb G$; $\phi$ is intrinsically differentiable $\mathcal{H}^k\llcorner\mathbb W$-almost everywhere on $U$, where $k:=\dim_{H}(\mathbb W)$; the map $\Phi$ is Pansu differentiable $\mathcal{H}^k\llcorner \mathbb W$-almost everywhere on $U$, and the following area formula holds: \begin{equation}\label{eqn:AreaElementFormula}
\mathcal{H}^k(\Phi(V))=\int_V J(d\Phi_x)d\mathcal{H}^k(x), \qquad \mbox{for every Borel set $V\subseteq U$},
\end{equation}
where $d\Phi_x$ is the Pansu differential of $\Phi$ at $x\in V$ and $J(d\Phi_x)$ is its Jacobian, see \eqref{def:definitionOfJ}.
\end{teorema}

We stress that the hypotheses of \cref{thm:AbelianRademacherIntro} are satisfied for example when $\mathbb W$ is horizontal. Thus our theorem holds for intriniscally Lipschitz horizontal surfaces in arbitrary Carnot groups. 
We stress also that an area formula for maps that parametrizes $C^1_{H}$-submanifolds, that is a more restrictive condition than being intrinsically Lipschitz, has been recently obtained in \cite[Theorem 1.1]{JNGV20}. In the latter reference the authors call {\em $C^1_{\mathrm H}$-submanifold} a subset of a Carnot group $\mathbb G$ that is, locally around every point, the zero-level set of a $C^1_{\mathrm H}$ function $f$ defined on an open subset of $\mathbb G$ with value in a Carnot group $\mathbb G'$, and such that moreover the Pansu differential $Df$ is surjective and $\mathrm{Ker}(Df)$ splits $\mathbb G$.

The proof of \cref{thm:AbelianRademacherIntro} is considerably simpler than the proof of a \textbf{low-codimensional} Rademacher theorem because in that case it may not be true, as in our case, that the intrinsic Lipschitz property of $\phi$ reads as the graph map $\Phi$ being Lipschitz (see \cref{rem:CounterexampleNormal}).
In fact we apply Pansu-Rademacher theorem
and Magnani's area formula
to deduce that the graph map $\Phi$ is Pansu differentiable almost everywhere and the area formula holds. Thus in order to complete the proof of \cref{thm:AbelianRademacherIntro}, we are left to relate the Pansu differentiability of the graph map $\Phi$ with the intrinsic differentiability of the map $\phi$, and this is done in \cref{thm:ToolForRademacher}.

\textbf{Acknowledgments}: The first author is partially supported by the European Research Council (ERC Starting Grant 713998 GeoMeG `\emph{Geometry of Metric Groups}'). The authors are grateful to Enrico Le Donne for several suggestions that led to an improvement of this note. They are also grateful to Sebastiano Don for precious comments.

\section{Preliminaries}\label{sec:Prel}
For the basic terminology on Carnot groups and graded groups we refer the reader to \cite{LD17}. Every Carnot group in this note will come together with a stratification $\mathrm{Lie}(\mathbb G)=V_1\oplus\dots\oplus V_s$. We recall some notation.  
Every graded group has a one-parameter family of dilations that we denote by $\{\delta_\lambda: \lambda >0\}$. 
We will indicate with $\delta_{\lambda}$ both the dilation of factor $\lambda$ on the group and its differential.
 
 Given a graded group $\mathbb G$, we fix a homogeneous norm $\|\cdot\|$ on $\mathbb G$, which is unique up to
 bi-Lipschitz equivalence. 
 Moreover, by \cite[Theorem 5.1]{step2} on every Carnot group there exists a slight variation of the anisotropic homogeneous norm that induces a left-invariant homogeneous distance $d$. \textbf{We eventually fix this particular homogeneous norm from now on}. We also recall that with {\em homogeneous homomorphism} we mean a homomorphism that commutes with $\delta_{\lambda}$, for every $\lambda>0$. We now deal with complementary subgroups and maps between them.
 
\begin{definizione}[Complementary subgroups]\label{def:ComplementarySubgroups}
Given a Carnot group $\mathbb G$ with identity $e$, we say that two subgroups $\mathbb W$ and $\mathbb L$ are \emph{complementary subgroups} in $\mathbb G$ if they are {\em homogeneous}, i.e., closed under the action of $\delta_{\lambda}$ for every $\lambda>0$, $\mathbb G=\mathbb W\cdot \mathbb L$, and $\mathbb W\cap \mathbb L=\{e\}$. 
\end{definizione}

Given two complementary subgroups $\mathbb W$ and $\mathbb L$ in a Carnot group $\mathbb G$, we denote the {\em projection maps} from $\mathbb G$ onto $\mathbb W$ and onto $\mathbb L$ by $\pi_{\mathbb W}$ and $\pi_{\mathbb L}$, respectively. Defining $g_{\mathbb W}:=\pi_{\mathbb W} g$ and $g_{\mathbb L} := \pi_{\mathbb L} g$ for any $g\in \mathbb G$, one has
\begin{equation*}\label{eqn:ComponentsSplitting}
g=(\pi_{\mathbb W} g)\cdot(\pi_{\mathbb L} g)= g_{\mathbb W}\cdot g_{\mathbb L}.
\end{equation*}
We recall the following basic terminology: a {\em horizontal subgroup} of a Carnot group $\mathbb G$ is a homogeneous subgroup of it that is contained in $\exp(V_1)$, where $V_1$ is the first layer of the stratification of $\mathrm{Lie}(\mathbb G)$; a {\em Carnot subgroup} $\mathbb W$ of a Carnot group $\mathbb G$ is a homogeneous subgroup of it such that the first layer $V_1(\mathbb W):=V_1\cap\mathrm{Lie}(\mathbb W)$ of the grading of $\mathrm{Lie}(\mathbb W)$ inherited from the stratification of $\mathrm{Lie}(\mathbb G)$ is the first layer of a stratification of $\mathrm{Lie}(\mathbb W)$. Before going on we add the following remark, which will be important in the proof of \cref{thm:AbelianRademacherIntro}.
\begin{osservazione}[The complement of a normal subgroup is Carnot]\label{rem:KEY}
We claim that if $\mathbb W$ and $\mathbb L$ are two homogeneous complementary subgroups of a Carnot group $\mathbb G$, and $\mathbb L$ is \textbf{normal}, then $\mathbb W$ is a \textbf{Carnot subgroup} of $\mathbb G$. First we claim that $\mathrm{Lie}(\mathbb L)$ is an ideal of $\mathrm{Lie}(\mathbb G)$. Indeed, if we set $\mathrm{Ad}$ to be the adjoint operator on $\mathbb G$, see \cite[page 12]{CG90}, for every $X\in\mathrm{Lie}(\mathbb G)$ and $Y\in \mathrm{Lie}(\mathbb L)$ we have $[X,Y]=\frac{d}{dt}_{|_{t=0}}\mathrm{Ad}_{\exp(tX)}Y\in\mathrm{Lie}(\mathbb L)$, where the last inclusion follows from the fact that $\mathrm{Ad}_{\exp(tX)}Y\in\mathrm{Lie}(\mathbb L)$ for every $t>0$, because $\mathbb L$ is normal. Therefore, if $\mathrm{Lie}(\mathbb G)=V_1\oplus\dots\oplus V_s$ is the stratification of $\mathrm{Lie}(\mathbb G)$, in order to prove that $\mathbb W$ is a Carnot subgroup of $\mathbb G$ it is enough to prove that 
\begin{equation}\label{eqn:Incl}
V_{i+1}\cap\mathrm{Lie}(\mathbb W)\subseteq [V_{1}\cap\mathrm{Lie}(\mathbb W),V_{i}\cap\mathrm{Lie}(\mathbb W)],
\end{equation} 
for every $i=1,\dots,s$; indeed the other inclusion $V_{i+1}\cap\mathrm{Lie}(\mathbb W)\supseteq [V_{1}\cap\mathrm{Lie}(\mathbb W),V_{i}\cap\mathrm{Lie}(\mathbb W)]$ follows from the fact that $\mathrm{Lie}(\mathbb W)$ is a Lie subalgebra of $\mathrm{Lie}(\mathbb G)$ because $\mathbb W$ is a homogeneous subgroup of $\mathbb G$. The inclusion \eqref{eqn:Incl} readily follows from the following observation, extended linearly: if $p_1\in V_1$ and $p_i\in V_i$ are such that $[p_1,p_i]\in\mathrm{Lie}(\mathbb W)$, then $[\pi_{\mathrm{Lie}(\mathbb W)}p_1,\pi_{\mathrm{Lie}(\mathbb W)}p_i]=[p_1,p_i]$, where $\pi_{\mathrm{Lie}(\mathbb W)}$ is the projection associated to the splitting $\mathrm{Lie}(\mathbb G)=\mathrm{Lie}(\mathbb W)\oplus\mathrm{Lie}(\mathbb L)$. Indeed, $[p_1,p_i]+\mathrm{Lie}(\mathbb L)=[\pi_{\mathrm{Lie}(\mathbb W)}p_1+\pi_{\mathrm{Lie}(\mathbb L)}p_1,\pi_{\mathrm{Lie}(\mathbb W)}p_i+\pi_{\mathrm{Lie}(\mathbb L)}p_i]+\mathrm{Lie}(\mathbb L)=[\pi_{\mathrm{Lie}(\mathbb W)}p_1,\pi_{\mathrm{Lie}(\mathbb W)}p_i]+\mathrm{Lie}(\mathbb L)$, since $\mathrm{Lie}(\mathbb L)$ is an ideal. From the latter equality the last claim follows since $[p_1,p_i]\in\mathrm{Lie}(\mathbb W)$. 

\end{osservazione}
We now describe what is the parametrizing function of some translation of the intrinsic graph of a function. \textbf{Unless anything different is declared, we eventually fix from now on two arbitrary complementary subgroups $\mathbb W$ and $\mathbb L$ of a Carnot group $\mathbb G$}.
\begin{definizione}[Intrinsic graph]\label{def:IntrinsicGraph}
    Given a function $\phi:U\subseteq \mathbb W\to\mathbb L$, we define the {\em intrinsic graph} of $\phi$ as follows:
    $$
    \mathrm{graph}(\phi):= \{\Phi(w):= w\cdot \phi(w): w\in U\}=:\Phi(U).
    $$
\end{definizione}

\begin{definizione}[Intrinsic translation of a function]\label{def:PhiQ}
	Given a function $\phi\colon U\subseteq\mathbb W\to\mathbb L$, we define, for every $q\in\mathbb G$,
	\[
	{U}_q:=\{a\in\mathbb W: \pi_{\mathbb W}(q^{-1}\cdot a)\in {U}\},
	\]
	and ${\phi}_q\colon{U}_q\subseteq \mathbb W\to\mathbb L$ by setting
	\begin{equation}\label{eqn:Phiq}
	{\phi}_q(a):=\big(
	\pi_{\mathbb L}(q^{-1}\cdot a)\big)^{-1}\cdot {\phi}\big(\pi_{\mathbb W}(q^{-1}\cdot a)\big).
	\end{equation}
\end{definizione}
For the following proposition we also refer the reader to \cite[Proposition 2.21 and Remark 2.23]{FranchiSerapioni16}.
\begin{proposizione}\label{prop:PropertiesOfIntrinsicTranslation}
	Given a function ${\phi}\colon{U}\subseteq\mathbb W\to\mathbb L$, and $q\in\mathbb G$, the following facts hold:
	\begin{itemize}
		\item[(a)] $\mbox{graph}({\phi}_q)=q\cdot\mbox{graph}({\phi})$;
		\item[(b)] if $\mathbb L$ is normal one gets that, for all $a\in U$,
		\begin{equation}\label{eqn:TransformationLNormal}
		    \pi_{\mathbb W}(q^{-1}\cdot a)=q_{\mathbb W}^{-1}\cdot a, \qquad \pi_{\mathbb L}(q^{-1}\cdot a)=a^{-1}\cdot q_{\mathbb W}\cdot q_{\mathbb L}^{-1}\cdot q_{\mathbb W}^{-1}\cdot a,
		\end{equation}
		and thus
		\begin{equation}\label{eqn:TransformationLNormal2}
		    {\phi}_q(a)=a^{-1}\cdot q_{\mathbb W} \cdot q_{\mathbb L}\cdot q_{\mathbb W}^{-1}\cdot a\cdot{\phi}(q_{\mathbb W}^{-1}\cdot a), \qquad  U_q = q_{\mathbb W}\cdot  U;
		\end{equation}
	\end{itemize}
\end{proposizione}
\begin{proof}
	The proof of (a) directly follows from \eqref{eqn:Phiq}, which yields
	\begin{equation}\label{eqn:GraphOfPhiQEqualsQGraphOfPhi}
	a\cdot{\phi}_q(a)=q\cdot\pi_{\mathbb W}(q^{-1}\cdot a)\cdot {\phi}(\pi_{\mathbb W}(q^{-1}\cdot a)),\qquad \forall a\in  U_q.
	\end{equation}
    Let us prove (b). Since $\mathbb L$ is normal, the following holds: $$a^{-1}\cdot q_{\mathbb W}\cdot q_{\mathbb L}^{-1}\cdot q_{\mathbb W}^{-1}\cdot a= a^{-1}\cdot q_{\mathbb W}\cdot q_{\mathbb L}^{-1}\cdot (a^{-1}\cdot q_{\mathbb W})^{-1}\in \mathbb L.$$
    Moreover it holds that $q_{\mathbb W}^{-1}\cdot a\in \mathbb W$, and since 
    $$
    q_{\mathbb W}^{-1}\cdot a\cdot a^{-1}\cdot q_{\mathbb W}\cdot q_{\mathbb L}^{-1}\cdot q_{\mathbb W}^{-1}\cdot a = q^{-1}\cdot a,
    $$
    the equation \eqref{eqn:TransformationLNormal} holds and then \eqref{eqn:TransformationLNormal2} is a consequence of \eqref{eqn:TransformationLNormal} and \eqref{eqn:Phiq}.
\end{proof}

For the notion of intrinsically Lipschitz function we refer the reader to \cite{FranchiSerapioni16}. We explicitly recall here, for the reader's benefit, only the notions of intrinsically linear and intrinsically differentiable functions. For these definitions, and the study of some properties related, we refer the reader to \cite{ASCV06, FSSC06, FSSC11, FMS14, SC16, ADDDLD20}. Eventually, in order to fix the notation, we recall the notion of Pansu differentiability, see also \cite{Pansu}.
\begin{definizione}[Intrinsically linear function]
    The map $\ell: \mathbb W \to \mathbb L$ is said to be {\em intrinsically linear} if $\mathrm{graph}(\ell)$ is a homogeneous subgroup of $\mathbb G$.
\end{definizione}
\begin{definizione}[Intrinsically differentiable function]\label{defiintrinsicdiff}
   Let ${\phi}\colon{U}\subseteq \mathbb W \to\mathbb L$ be a function with $U$ \textbf{Borel} in $\mathbb W$. Fix a \textbf{density point}\footnote{If $k:=\mathrm{dim}_H(\mathbb W)$ is the Hausdorff dimension of $\mathbb W$, we say that $a_0\in U$ is a {\em density point} of $U$, and we write $a_0\in\mathcal D(U)$, if \\ $\mathcal{H}^k\llcorner \mathbb{W}(U\cap B(a_0,r))/\mathcal{H}^k\llcorner \mathbb{W}(B(a_0,r))\to 1$ as $r$ goes to $0$.} $a_0\in\mathcal D({U})$ of $U$, let $p_0:={\phi}(a_0)^{-1}\cdot a_0^{-1}$ and denote by ${\phi}_{p_0}\colon {U}_{p_0}\subseteq \mathbb W\to\mathbb L$ the shifted function introduced in \cref{def:PhiQ}. We say that ${\phi}$ is {\em intrinsically differentiable} at $a_0$
	if there is an intrinsically linear map $d^{\phi}\phi_{a_0}\colon\mathbb W\to\mathbb L$ such that
	\begin{equation}\label{eqn:IdInCoordinates2}
	\lim_{b\to e,\, b\in U_{p_0}}\frac{\|d^{\phi}\phi_{a_0}[b]^{-1}\cdot{\phi}_{p_0}(b)\|}{\|b\|}= 0.
	\end{equation}
	The function $d^{\phi}\phi_{a_0}$ is called the {\em intrinsic differential} of $\phi$ at $a_0$. 
	Notice that one can take the limit in \eqref{eqn:IdInCoordinates2} because, if $a_0\in \mathcal D( U)$, then $e \in \mathcal D( U_{p_0})$. This last claim comes from the invariant properties of \cref{prop:PropertiesOfIntrinsicTranslation}.
\end{definizione}

\begin{osservazione}[Intrinsic differentiability and tangent subgroups]\label{rem:TangentSubgroups}
Whenever the intrinsic differential introduced in \cref{defiintrinsicdiff} exists, it is unique: see \cite[Theorem 3.2.8]{FMS14}. Let us recall the notion of tangent subgroup to the graph of a function.
If we fix $\phi: U\subseteq \mathbb W \to \mathbb L$, 
we say  that a homogeneous subgroup $\mathbb T$ of $\mathbb G$ is a {\em tangent subgroup} to $\mathrm{graph} (\phi)$ at $a_0\cdot\phi(a_0)$ if the following facts hold:
\begin{enumerate}
\item[(a)] $\mathbb T$ is a complementary subgroup of $\mathbb L$;
\item[(b)] The limit
\begin{equation*}
\lim_{\lambda \to \infty } \delta _\lambda \left((a_0\cdot \phi(a_0))^{-1}\cdot\mathrm{graph}(\phi) \right)  =\mathbb T,
\end{equation*}
holds in the sense of Hausdorff convergence on compact subsets of $\mathbb G$. 
\end{enumerate}
We notice that what the authors prove in \cite[Theorem 3.2.8]{FMS14} is the following: a function $\phi:U\subseteq \mathbb W\to \mathbb L$, with $U$ \textbf{open}, is intrinsically differentiable at $a_0$ if and only if $\mathrm{graph}(\phi)$ has a tangent subgroup $\mathbb T$ at $a_0\cdot\phi(a_0)$ and moreover $\mathbb T=\mathrm{graph}(d^\phi \phi_{a_0})$. 
\end{osservazione}
\begin{definizione}[Pansu differentiability]\label{def:Pansudiff}
Let $\mathbb W$ and $\mathbb G$ be two arbitrary \textbf{graded groups} endowed with two homogeneous left-invariant distances $d_{\mathbb W}$ and $d_{\mathbb G}$, respectively. Given $f:U\subseteq \mathbb W\to\mathbb G$ with $U$ \textbf{Borel}  and a \textbf{density point} $a_0\in\mathcal{D}(U)$, we say that $f$ is {\em Pansu differentiable} at $a_0$ if there exists a homogeneous homomorphism $df_{a_0}:\mathbb W\to\mathbb G$, that we call {\em Pansu differential} at $a_0$, such that
$$
\lim_{a\to a_0\, a\in U}\frac{d_{\mathbb G}(f(a_0)^{-1}\cdot f(a),df_{a_0}[a_0^{-1}\cdot a])}{d_{\mathbb W}(a,a_0)}=0.
$$
\end{definizione}

\section{Proof of the theorem}\label{sec:Theorem}
In what follows we prove that, in case $\mathbb L$ is \textbf{normal}, the intrinsic differentiability of a function $\phi$ can be read as the Pansu differentiability of the graph map $\Phi$. For the forthcoming lemma we also refer the reader to \cite[second item of Corollary 3.1.4]{FMS14}.
\begin{lemma}\label{prop:PansuDifferentiabilityIntrinsicLinearLNormal} 
    Let $\mathbb W$ and $\mathbb L$ be two complementary subgroups of a Carnot group $\mathbb G$, with $\mathbb L$ \textbf{normal}. The map $\ell:\mathbb W\to\mathbb L$ is intrinsically linear if and only if the graph map of $\ell$, i.e., $L:\mathbb W\to\mathbb G$ defined as $L(w):=w\cdot \ell(w)$, is a homogeneous homomorphism. 
\end{lemma}

The forthcoming proposition is inspired by \cite[Proposition 3.25(i)]{ArenaSerapioni}. We give a detailed proof in our context, since the last part of the argument, i.e., the part in which we invoke the forthcoming \cref{lem:LemmaBuono}, is different with respect to the reference.
\begin{proposizione}\label{thm:ToolForRademacher}
Let $\mathbb W$ and $\mathbb L$ be two complementary subgroups of a Carnot group $\mathbb G$, with $\mathbb L$ \textbf{normal}. Let $\phi: U\subseteq \mathbb W\to \mathbb L$ be a function with $U$ \textbf{Borel}. Given a \textbf{density point} $a_0\in\mathcal D(U)$, we have that $\phi$ is intrinsically differentiable at $a_0$ if and only if the graph map $\Phi: U\subseteq \mathbb W\to \mathbb G$ is Pansu differentiable at $a_0$. Moreover if any of the previous two holds we have the following formula:
$$
d\Phi_{a_0}[w]=w\cdot d^{\phi}\phi_{a_0}[w], \qquad \forall w\in\mathbb W.
$$
\end{proposizione}
\begin{proof}
Let us first notice that $p_0:=\phi(a_0)^{-1}\cdot a_0^{-1}=a_0^{-1}\cdot a_0\cdot \phi(a_0)^{-1}\cdot a_0^{-1}$, and thus, from the fact that $\mathbb L$ is normal,
\begin{equation}\label{eqn:ProjectionP0}
(p_0)_{\mathbb W}=a_0^{-1}, \qquad (p_0)_{\mathbb L}=a_0\cdot\phi(a_0)^{-1}\cdot a_0^{-1}.
\end{equation}

Let us assume that $\phi$ is intrinsically differentiable at $a_0$. We are going to prove that the graph map $\Phi$ is Pansu differentiable at $a_0$. From the intrinsic differentiability of $\phi$, we know that there exists the intrinsic differential $d^{\phi}\phi_{a_0}$ as in \cref{defiintrinsicdiff}, which is an intrinsically linear map by definition. We define its graph map
\begin{equation}\label{eqn:DefinitionPansuInProof}
d\Phi_{a_0}[w]:=w\cdot d^{\phi}\phi_{a_0}[w], \qquad \forall w\in \mathbb W.
\end{equation}
From \cref{prop:PansuDifferentiabilityIntrinsicLinearLNormal}, it follows that the map $d\Phi_{a_0}$ is a homogeneous homomorphism. We show that it is the Pansu differential of the graph map $\Phi$. Indeed, let us take $w\in U$ and compute
\begin{equation}\label{eqn:LongComputation}
    \begin{split}
        d\Phi_{a_0}[a_0^{-1}\cdot w]^{-1}\cdot \Phi(a_0)^{-1}\cdot\Phi(w) &= d^{\phi}\phi_{a_0}[a_0^{-1}\cdot w]^{-1}\cdot w^{-1}\cdot a_0\cdot \phi (a_0)^{-1}\cdot a_0^{-1}\cdot w\cdot \phi (w)= \\
        &= d^{\phi}\phi_{a_0}[a_0^{-1}\cdot w]^{-1}\cdot  \phi_{p_0}(a_0^{-1}\cdot w),
    \end{split}
\end{equation}
where in the first equality we used the definition \eqref{eqn:DefinitionPansuInProof} and in the second one we used \eqref{eqn:ProjectionP0} and the explicit expression in the first equality of \eqref{eqn:TransformationLNormal2}. Notice that, from the second equality of \eqref{eqn:TransformationLNormal2} and the first equality of \eqref{eqn:ProjectionP0}, we get that $ U_{p_0}=a_0^{-1}\cdot  U$. Thus \eqref{eqn:LongComputation}, jointly with the intrinsic differentiability of $\phi$, see \eqref{eqn:IdInCoordinates2}, tells us that 
$$
\lim_{w\to a_0,\, w\in U} \frac{\|d\Phi_{a_0}[a_0^{-1}\cdot w]^{-1}\cdot \Phi(a_0)^{-1}\cdot\Phi(w)\|}{\|a_0^{-1}\cdot w\|} = 0,
$$
that is the Pansu differentiability of $\Phi$ at $a_0$  with differential $d\Phi_{a_0}$. 

Vice versa, let us assume that the graph map $\Phi$ is Pansu differentiable at $a_0$. Thus there exists a homogeneous homomorphism $d\Phi_{a_0}:\mathbb W\to\mathbb G$ such that 
\begin{equation}\label{eqn:PansuDiff}
\lim_{w\to a_0,\, w\in U} \frac{\|d\Phi_{a_0}[a_0^{-1}\cdot w]^{-1}\cdot \Phi(a_0)^{-1}\cdot\Phi(w)\|}{\|a_0^{-1}\cdot w\|} = 0.
\end{equation}
By using the fact that $\mathbb L$ is normal, and by simple computations, we get that
$$
\pi_{\mathbb W}\big(d\Phi_{a_0}[a_0^{-1}\cdot w]^{-1}\cdot \Phi(a_0)^{-1}\cdot\Phi(w)\big)=\big((d\Phi_{a_0}[a_0^{-1}\cdot w])_{\mathbb W}\big)^{-1}\cdot a_0^{-1}\cdot w, \qquad \forall w\in  U.
$$
Now notice that since there exists a geometric constant $C>0$ such that $\|g_{\mathbb W}\|\leq C\|g\|$ for every $g\in\mathbb G$ (see \cite[Proposition 2.12]{FranchiSerapioni16}), from the previous equality and \eqref{eqn:PansuDiff} we deduce
\begin{equation}\label{eqn:ProjectedDifferentiability}
\lim_{w\to a_0,\, w\in U} \frac{\left\|\big((d\Phi_{a_0}[a_0^{-1}\cdot w])_{\mathbb W}\big)^{-1}\cdot a_0^{-1}\cdot w\right\|}{\|a_0^{-1}\cdot w\|} = 0.
\end{equation}
Since $d\Phi_{a_0}$ is a homogeneous homomorphism, the limit in \eqref{eqn:ProjectedDifferentiability} allows us to use the forthcoming \cref{lem:LemmaBuono} with $U_e:=a_0^{-1}\cdot U$. This leads to the equality $\pi_{\mathbb W}\circ (d\Phi_{a_0})= (\mathrm{id})_{|_{\mathbb W}}$. 

Thus, by using the splitting, for every $w\in \mathbb W$ it holds $d\Phi_{a_0}[w]=:w\cdot d^{\phi}\phi_{a_0}[w]$ for some map $d^{\phi}\phi_{a_0}:\mathbb W\to \mathbb L$. Now we are in a position to apply \cref{prop:PansuDifferentiabilityIntrinsicLinearLNormal} to deduce that $d^{\phi}\phi_{a_0}$ is intrinsically linear, because its graph is a homogeneous homomorphism being a Pansu differential. Now the intrinsic differentiability of $\phi$ at $a_0$, with intrinsic differential $d^{\phi}\phi_{a_0}$, follows by joining the computations in \eqref{eqn:LongComputation} with \eqref{eqn:PansuDiff}. 
\end{proof}
\begin{lemma}\label{lem:LemmaBuono}
Let $\mathbb W$ and $\mathbb L$ be two complementary subgroups of a Carnot group $\mathbb G$ and let $U_e\subseteq \mathbb W$ be a \textbf{Borel} set containing the identity $e$ and for which $e$ is a \textbf{density point}. Let us assume $F:\mathbb W\to\mathbb G$ is a continuous homogeneous map, i.e., it commutes with $\delta_{\lambda}$ for every $\lambda>0$. Let us further assume that 
$$
\lim_{w\to e,\, w\in U_e}\frac{\left\|(F(w)_{\mathbb W})^{-1}\cdot w\right\|}{\|w\|}=0.
$$
Then $F(w)_{\mathbb W}=w$ for every $w\in \mathbb W$.
\end{lemma}
\begin{proof}
Notice first that the map $w\to (F(w)_{\mathbb W})^{-1}$ is a homogeneous map.  
Indeed, from the homogeneity of $F$, the homogeneity of the projection onto $\mathbb W$ and the homogeneity of the inverse, respectively, we get
\begin{equation}\label{eqn:Hom}
\left((F(\delta_{\lambda} w))_{\mathbb W}\right)^{-1}=\left((\delta_{\lambda}F(w)\right)_{\mathbb W})^{-1}=(\delta_{\lambda}(F(w)_{\mathbb W}))^{-1}=\delta_{\lambda}((F(w)_{\mathbb W}))^{-1}, \qquad \forall \lambda>0.
\end{equation}
Set $S^{\mathbb W}:=\{w\in \mathbb W:\|w\|=1\}$. Since $e$ is a density point of $U_e\subseteq \mathbb W$, the following holds: there exists $D\subseteq S^{\mathbb W}$, dense in $S^{\mathbb W}$, such that for every $w\in D$ there exists an infinitesimal sequence $\{t_j\}_{j\in\mathbb N}$ such that $\delta_{t_j}w\in U_e$ for all $j\in\mathbb N$.\footnote{Indeed if $U\subseteq \mathbb W$ is Borel and for $\mathcal{H}^k\llcorner \mathbb W$-almost every point $x$ in $U$ there exist a sequence $h_j(x)$ converging to $e$ and $0<\lambda(x)<1$ with $B(x\cdot h_j(x),\lambda(x)\|h_j(x)\|)\cap E=\emptyset$ for all $j\in\mathbb N$, then $\mathcal{H}^k\llcorner\mathbb W(E)=0$.}

Fix $w\in D$. We claim that $F(w)_{\mathbb W}=w$ so that by density of $D$ and by the continuity and the homogeneity of $F$ the thesis follows. Indeed, let us fix $\epsilon>0$. By hypothesis, and since there exists an infinitesimal sequence $\{t_j\}_{j\in\mathbb N}$ such that $\delta_{t_j}w\in U_e$, we get that there is $t_{j_0}>0$ such that 
$$
\|(F(\delta_{t_{j_0}}w)_{\mathbb W})^{-1}\cdot \delta_{t_{j_0}}w\| \leq \epsilon \|\delta_{t_{j_0}}w\|. 
$$
By the homogeneity in \eqref{eqn:Hom}, the homogeneity of the norm and the fact that $\|w\|=1$, since $w\in D\subseteq S^{\mathbb W}$, we get
$$
\|(F(w)_{\mathbb W})^{-1}\cdot w\| \leq \epsilon. 
$$
Thus from the fact that $\epsilon>0$ is arbitrary we get the sought conclusion.
\end{proof}

We are now in a position to prove our \cref{thm:AbelianRademacherIntro}.
Let us recall the definition of the Jacobian of a homogeneous map. Take a \textbf{homogeneous map} $F:\mathbb W\to \mathbb G$ between \textbf{graded groups} $\mathbb W$ and $\mathbb G$, equipped with homogeneous left-invariant distances $d_{\mathbb W}$ and $d_{\mathbb G}$, respectively. Denote by $k:=\mathrm{dim}_{H}(\mathbb W)$ the Hausdorff dimension of $\mathbb W$. The {\em Jacobian} of $F$ is
\begin{equation}\label{def:definitionOfJ}
J(F):=\frac{\mathcal{H}^k(F(B(e,1)))}{\mathcal{H}^k(B(e,1))},
\end{equation}
where $B(e,1)$ is the ball centered at the identity $e$ of $\mathbb W$, and of radius 1. 

\begin{proof}[Proof of \cref{thm:AbelianRademacherIntro}]
Since $\mathbb L$ is \textbf{normal}, we can use \cite[Proposition 3.7]{FranchiSerapioni16} and thus, from the fact that $\phi$ is intrinsically Lipschitz, we deduce that the graph map $\Phi: U\subseteq \mathbb W\to \mathbb G$ is Lipschitz. Since $\mathbb W$ is a \textbf{Carnot subgroup}, due to \cref{rem:KEY}, we are in a position to apply Rademacher theorem (see \cite{Pansu} and \cite[Theorem 3.9]{Magnani01}) to the graph map $\Phi: U\subseteq \mathbb W\to\mathbb G$, in order to conclude that it is $\mathcal H^k\llcorner \mathbb W$-almost everywhere Pansu differentiable on $U$. Eventually, we apply \cref{thm:ToolForRademacher} to conclude that every point of Pansu differentiability of $\Phi$ is a point of intrinsic differentiability of $\phi$. Finally, the area formula \eqref{eqn:AreaElementFormula} is a direct consequence of the area formula in \cite[Theorem 4.3.4]{MagnaniPhD} applied to the graph map $\Phi$, after having noticed that $\Phi:U\subseteq \mathbb W\to \mathbb G$ is injective. 
\end{proof}

\begin{osservazione}
By joining the result of \cref{thm:AbelianRademacherIntro} and \cref{rem:TangentSubgroups} we conclude that in the hypotheses of \cref{thm:AbelianRademacherIntro} the intrinsically Lipschitz property guarantees the existence, at $(\Phi)_*(\mathcal{H}^k\llcorner U)$-almost every point on the graph of $\phi$, of a tangent subgroup.
Moreover, from \cref{thm:ToolForRademacher} we get that, whenever the Pansu differential $d\Phi_x$ exists, then $\phi$ is intrinsically differentiable at $x$ and $d\Phi_x[w]=w\cdot d^{\phi}\phi_{x}[w]$ for all $w\in\mathbb W$.
Taking into account this equality, the definition of the Jacobian \eqref{def:definitionOfJ} and \cref{rem:TangentSubgroups} we stress that the area element $J(d\Phi_x)$ in the area formula \eqref{eqn:AreaElementFormula} only depends on the geometry of the tangent subgroup of $\mathrm{graph}(\phi)$ at $\Phi(x)$, which is $\mathrm{graph}(d^{\phi}\phi_x)$.
\end{osservazione}
\begin{osservazione}
The hypotheses of \cref{thm:AbelianRademacherIntro} are satisfied whenever we take an intrinsically Lipschitz function $\phi:U\subseteq \mathbb W\to\mathbb L$, with $\mathbb W$ \textbf{horizontal}. Thus our result applies in particular to intrinsically Lipschitz horizontal surfaces in arbitrary Carnot groups.
\end{osservazione}
\begin{osservazione}\label{rem:CounterexampleNormal}
If we do not assume $\mathbb L$ to be \textbf{normal} in the hypotheses of \cref{thm:AbelianRademacherIntro}, but still we assume that $\mathbb W$ is a \textbf{Carnot subgroup}, the graph map $\Phi:U\subseteq\mathbb W\to\mathbb L$ may not be Lipschitz when $\phi$ is intrinsically Lipschitz. The forthcoming example is also found in \cite{FranchiSerapioni16}.

Indeed, let us take the second Heisenberg group $\mathbb H^2$, with a basis of the Lie algebra given by $(X_1,X_2,X_3,X_4,X_5)$, where the only nontrivial relations are $[X_1,X_3]=[X_2,X_4]=X_5$. Let us identify $\mathbb H^2$ with $\mathbb R^5$ by means of exponential coordinates of the first kind. Set, in those exponential coordinates, $\mathbb W:=\{(0,x_2,x_3,x_4,x_5):x_2,x_3,x_4,x_5\in \mathbb R\}$ and $\mathbb L:=\{(x_1,0,0,0,0):x_1\in \mathbb R\}$. We notice that $\mathbb W$ is a \textbf{Carnot subgroup} and $\mathbb L$ is \textbf{not normal}. 

It is easily verified that the map $\phi:\mathbb W\to \mathbb L$ defined as $\phi(0,x_2,x_3,x_4,x_5):=(1,0,0,0,0)$ for every $(x_2,x_3,x_4,x_5)\in \mathbb R^4$ is intrinsically Lipschitz. Moreover, if we fix $\epsilon>0$ we have that $\Phi(0,0,\epsilon,0,0)=(1,0,\epsilon,0,-\epsilon/2)$ and $\Phi(0,0,0,0,0)=(1,0,0,0,0)$. Thus $$\|\Phi(0,0,0,0,0)^{-1}\cdot\Phi(0,0,\epsilon,0,0)\|=\|(0,0,\epsilon,0,-\epsilon)\|\sim_{\epsilon\to 0} \epsilon^{1/2},$$ and then $\Phi$ cannot be Lipschitz, since $\|(0,0,\epsilon,0,0)\|\sim_{\epsilon\to 0} \epsilon$. 
\end{osservazione}

\printbibliography

\end{document}

%% file: structure.tex
\usepackage[
nochapters, % Turn off chapters since this is an article        
pdfspacing, % Makes use of pdftex’ letter spacing capabilities via the microtype package
dottedtoc % Dotted lines leading to the page numbers in the table of contents
]{classicthesis} % The layout is based on the Classic Thesis style
\usepackage{amsopn}
\usepackage[T1]{fontenc} % Use 8-bit encoding that has 256 glyphs
\usepackage{multirow}
\usepackage[utf8]{inputenc} % Required for including letters with accents
\usepackage{hhline}
\usepackage{graphicx} % Required for including images
\graphicspath{{Figures/}} % Set the default folder for images
\usepackage{indentfirst}
\usepackage{enumitem} % Required for manipulating the whitespace between and within lists
\usepackage{savesym}
\usepackage{mathrsfs}
\usepackage{geometry}
\usepackage[
    backend=biber,
    style=numeric,
    natbib=false,
    url=false, 
    doi=false,
    eprint=false,
    maxnames=50
]{biblatex}

\usepackage{subfig} % Required for creating figures with multiple parts (subfigures)

\usepackage{amsmath, amssymb,amsthm,amsfonts} % For including math equations, theorems, symbols, etc

\usepackage[english,capitalize]{cleveref}

\usepackage{thmtools}
\usepackage{thm-restate}

\usepackage{hyperref}
\newcommand{\vertiii}[1]{{\left\vert\kern-0.25ex\left\vert\kern-0.25ex\left\vert #1 
    \right\vert\kern-0.25ex\right\vert\kern-0.25ex\right\vert}}

\theoremstyle{plain}
\newtheorem{teorema}{Theorem}[section]
\newtheorem{proposizione}[teorema]{Proposition}
\newtheorem{lemma}[teorema]{Lemma}

\newtheorem*{theorem*}{Theorem}

\theoremstyle{definition}
\newtheorem{definizione}{Definition}[section]

\theoremstyle{remark}
\newtheorem{osservazione}{Remark}[section]

\newcommand{\res}

\newcounter{const}

\newcounter{eps}

\hypersetup{
colorlinks=true, breaklinks=true,bookmarksnumbered,
urlcolor=webbrown, linkcolor=RoyalBlue, citecolor=webgreen, % Link colors
pdftitle={}, % PDF title
pdfauthor={\textcopyright}, % PDF Author
pdfsubject={}, % PDF Subject
pdfkeywords={}, % PDF Keywords
pdfcreator={pdfLaTeX}, % PDF Creator
pdfproducer={LaTeX with hyperref and ClassicThesis} % PDF producer
}